\documentclass[a4paper,fleqn]{elsarticle}
\usepackage{amsfonts,amssymb,amsmath,amsbsy,amsthm,amscd}
\usepackage{lineno,hyperref}
\usepackage[english]{babel}
\usepackage{graphicx}
\usepackage{color}
\usepackage[cp1251]{inputenc}
\usepackage{graphicx}			
\usepackage{wrapfig}			

\usepackage{geometry}			
\usepackage{indentfirst}		
\usepackage{multicol}

\usepackage{graphicx}
\usepackage{color}
\usepackage{array}

\UseRawInputEncoding

\newtheorem{theorem}{Theorem}[section]

\newtheorem{lemma}[theorem]{Lemma}
\newtheorem{corollary}{Corollary}

\renewcommand{\le}{\leqslant}
\renewcommand{\ge}{\geqslant}

\begin{document}

\begin{frontmatter}

\title
{A model of a multiphase medium based on the closure of moment chains for the Vlasov-Poisson equations
}

\author[1]{Olga S. Rozanova} 
\ead{rozanova@mech.math.msu.su} 
\address[1]{
 Lomonosov Moscow
State University, Moscow 119991 Russia}
\begin{abstract}
We propose a method for closing moment chains for the Vlasov-Poisson kinetic system (the Landau fluid model), which yields a hyperbolic system of  equations at each subsequent step. This system is interpreted as a multicomponent medium, where each component has its own density, velocity, and pressure. The coupling between the components is mediated by an electric field. The principle of closing moment chains is that the final component of the fluid is assumed to be pressureless.
 \end{abstract}


\end{frontmatter}


\numberwithin{theorem}{section}
\numberwithin{remark}{section}

\section{Introduction}

The transition from the Euler-Poisson kinetic model, which describes the motion of a collisionless plasma, to a hydrodynamic model is one of the most intriguing problems in modeling an electron gas. On the one hand, such a transition seems entirely natural, as it involves the introduction of velocity-averaged quantities with a direct physical meaning -- the moments of the particle velocity and coordinate distribution function. The zeroth, first, and second-order  moments define the density, velocity, and pressure. The resulting equations have the form of conservation laws, but are not closed, meaning they contain a third centered moment, representing heat flux. They are much simpler to calculate than the original model, so it is very tempting to understand how to correctly formulate a hydrodynamic model, that is, express the third moment in terms of all the others while preserving the properties of the original kinetic model.
However, it is generally believed that the hydrodynamic model is incapable of reproducing the properties of the kinetic model, such as Landau damping, turbulence,  or the phenomenon of the loss of smoothness of the solution to the Cauchy problem. Therefore, traditional efforts are focused on selecting a closure that correctly reproduces certain properties of the kinetic model. Recently, methods that perform selection using machine learning and neural networks have been developed \cite{Bois}, \cite{Huang}. A recent review of both traditional and modern methods can be found in \cite{Burles}. It would seem natural to extend the chain of moments, that is, to express not the third moment, but subsequent moments, in terms of the first moments. However, this approach poses a problem associated with the loss of hyperbolicity of the system \cite{Cai}. In other words, the system resulting after closure does not ensure the finite velocity of propagation of disturbances inherent in the kinetic model.

In the present work, the problem of deriving a hyperbolic system of equations from a non-hyperbolic problem is analyzed from a purely mathematical standpoint. However, interest in this problem arises primarily from the requirements of numerical methods, since integrating a hyperbolic system demands significantly less computational power than integrating the full kinetic model. There are known moment-closure techniques that allow for the transition from the Vlasov equation to a hyperbolic system of equations. In the Levermore's method, the moment hierarchy  obtained by taking velocity moments of the one-dimensional kinetic equation under the assumption that the velocity distribution is a maximum-entropy function \cite{Levermore}, \cite{Porteous}. The hyperbolic quadrature method of moments (HyQMOM) is
 leads to a strictly hyperbolic system of equations  \cite{Fox}. In this case, the hyperbolic hierarchy for closures with even numbers of moments \cite{Morin} is build on the representation of the velocity distribution function  as a weighted sum of Dirac deltas.
Let us mention studies that address this problem in one way or another \cite{Cai}, \cite{Issan}, \cite{Koellermeier}, \cite{Burby}, \cite{Huang}.
There are limitations to the use of all these methods.


In this work, we propose an alternative conservative moment closures for the one-dimensional collisionless Vlasov-Poisson equations (the Landau-fluid model).
We   make no claim that this approach can completely solve the problem of transition from a kinetic to a hydrodynamic model. However, in our view, it does have some useful features.

1. It is significantly simpler than known methods.

2. It can be closed at any odd step and has the same properties for each closure: namely, it is hyperbolic (in the non-strict sense) and its coefficient matrix has two distinct eigenvalues, one of which is zero.


3. Closing the moment chain does not mean that all subsequent equations are discarded. On the contrary, subsequent equations become consequences of the previous ones.

The principle of chain closure is based on achieving equality in the H\"older inequality, meaning it is purely mathematical. However, it can be given a physical meaning.
Specifically, the system of moments can be represented as a system of balance equations describing a medium consisting of several naturally occurring phases, the interaction between which occurs through the electric field. These are not the phases that are usually discussed in the context of plasma equations, namely electrons, ions and neutral particles. For each phase we introduce higher-order densities and velocities  and obtain pairs of  equations of mass balance and linear momentum. The first velocity and density are conventional physical ones, while the others are artificially created and are called velocities and densities only because they are included in equations similar to the standard continuity equation and the momentum equation. Each phase is described by a pair of equations. Each phase depends on pressure, except the last. The principle of chain closure is that the last phase is assumed to be pressure-free.

Closure in our model does not mean discarding the entire infinite sequence of equations for moments of higher levels, but only that they turn out to be consequences of the previous ones. Thus, it can be assumed that the multiphase medium created in the described manner consists of an infinite number of phases, but all of them, starting from the one on which the closure was carried out, are pressureless.
\section{Construction of moment equations }
Let $F(t,x,v)\ge 0$ be
the phase space distribution function of
 electrons.
 We consider the kinetic Vlasov-Poisson equations for a collisionless electron plasma in its {\it dimensionless form} in the one-dimensional case
\begin{equation}\label{VP}
F_t +v F_x -  E(t,x) F_v=0, \qquad E_x= 1-n, \qquad n(t,x)=\int\limits_{\mathbb R} F dv,\quad x,v\in \mathbb R, \, t\ge 0,
\end{equation}
where the electric field strength $E$. The issues of existence of a solution to the Cauchy problem for such a system have been well studied i.e. \cite{Iord}, but we will not touch on them here.

Introduce the Hamburger moments:
$$
M_i(t,x)= \int\limits_{\mathbb R} v^i F(t,x,v) dv, \qquad k\in 0\cup \mathbb N,
$$
and assume that moments of any order exist. 
Obviously, $n=M_0$.


From the first equation \eqref{VP} it follows (after multiplying by $v^i$ and integrating over $v$):
\begin{equation}\label{1}
(M_0)_t + (M_{1})_x =0, \qquad (M_i)_t + (M_{i+1})_x = - i E M_{i-1}, i\in \mathbb N,
\end{equation}
(this system was obtained also in \cite{Jin}).
Let us fix $j\in \mathbb N$ and  denote
$\lambda_{j}(t,x)= ( M^2_j/M_{j-1}-M_{j+1})_x,$
then the  equations \eqref{1} for $M_j$ can be rewritten as
\begin{equation*}\label{3}
(M_j)_t + ( M^2_j/M_{j-1})_x = -j E M_{j-1}+ \lambda_j.
\end{equation*}
Since we want to avoid the denominators becoming zero, we will only close the chain {for odd $j=2k-1$}, $k\in\mathbb N$.


From the H\"older inequality it follows that $(M_{2k-1})^2\le M_{2k} M_{2k-2}$,  $k\in \mathbb N $.
Note that in a particular case
$F(t,x,v)=P(t,x)\,\delta(v-v_0)$ the inequality becomes equality.

{\it The principle of cutting off the chain and closing the system} is that at step $j=2k-1$ is the assumption that on this step the H\"older inequality become equality:
\begin{equation}\label{principle}
  (M_{2k-1})^2= M_{2k} M_{2k-2},\quad \mbox{or }  \quad  \lambda_{2k-1}=0.
\end{equation}

Further, if we know $M_i$, $i=0,\dots, n, \quad j=2k-1$, we can define subsequent moments in the same principle, i.e.
\begin{equation}\label{Mn+k}
M_{j+i}=\frac{M^{j+1}_j}{M_{j-1}^i}, \qquad i\in \mathbb N.
\end{equation}
Thus, we can see that the equations for $M_{j+1}$,  $M_{j+2}$, etc, are a consequence of the system for moments up to order $j$.

The closure of the system \eqref{1} at step $j$ has the form
\begin{eqnarray}\label{4}
& M_t + A(M) M_x= B,\\
& M=(M_0,..., M_j)^T, \quad B=(0,-EM_0, -2 EM_1,...,-j E M_{j-1})^T.\nonumber
\end{eqnarray}
\begin{lemma}
The structure of the matrix $A=A(M_0, M_1,..., M_j)$ of dimension $(j+1)\times (j+1)$ is the following. It  has $j$ rows consisting of one at the $i+2$ place in the $i$-th row and zeros at the remaining places, $i=0, 1,...,j-1$,
and the last, $j+1$-th row, consisting of zeros, except for the elements $A_{j \, j-1}= -\frac{M_j^2}{M_{j-1}^2}$ and $A_{j \, j}= \frac{2 M_j}{M_{j-1}}$.

The spectrum of the matrix $A$ is real and consists of the value
  $\frac{M_j}{M_{j-1}}$ of multiplicity 2 and the zero value of multiplicity $j-1$.
\end{lemma}
The {\it proof} is a direct computation. Note that \eqref{1} is a non-strictly hyperbolic system, it has no complete set of eigenvectors.
For $j=1$ the spectrum of $A$ has no zeros. $\Box$

\section{Multifluid interpretation}\label{MFI}
\begin{theorem}
For classical solutions system \eqref{4} can be rewritten in conservative form as
\begin{eqnarray}\nonumber
&&(n_1)_t+(n_1 U_1)_x=0,\qquad\qquad\qquad\qquad\qquad\,\,
(n_1 U_1)_t+(n_1 U_1^2 +P_1)_x=- E n_1, \nonumber
\\
&&(n_2)_t+(n_2 U_2)_x=- 2 E n_1 U_1,\qquad\qquad\qquad\,\,\,
(n_2 U_2)_t+(n_2 U_2^2 +P_2)_x=- 3 E n_2, \nonumber
\\
&&\dots \nonumber\\
&&(n_k)_t+(n_k U_k)_x=- (2k-2) E n_{k-1} U_{k-1},\quad\,
(n_k U_k)_t+(n_k U_k^2 +{{P_k}})_x=- (2k-1) E n_{k},\nonumber
\end{eqnarray}
denoted as (MF),
supplemented by
\begin{eqnarray}\label{E}
&&E_t + U_1 E_x= U_1.
\end{eqnarray}
Here
\begin{equation}\label{nUP_MF}
n_i=M_{2i-2},\qquad U_i=\frac{M_{2i-1}}{M_{2i-2}}, \qquad P_i=n_{i+1}-n_{i} U_{i}^2, \quad i=1,\dots, k, \quad k \in \mathbb R.
\end{equation}
The closure principle \eqref{principle} corresponds to $P_k=0$ on the step $k$.

The assumption $P_i=0$, $i>k$ makes an infinite system of pairs of equations describing fluids with the number $i$, a consequence of the closure of the system at step $k$.

\end{theorem}
\proof System (MF) can be obtained from \eqref{1} by a direct computation. The usual density and velocity  are $n_1=n$ and $U_1=U$. For smooth solutions equation \eqref{E} follows from the second equation of \eqref{VP}  and the first equation of (MF), the continuity equation $(n_1)_t+(n_1 U_1)_x=0$, assuming that $E$ is zero at infinity.
Further, assumption $P_i=0$, $i>k$ is equivalent to \eqref{Mn+k}.
$\Box$

Note that if
$
F(t,x,v)=P(t,x)\, \delta(v-v_0(t,x)),
$
then
$U_k=U_{1}=v_0,k\in \mathbb N, $
otherwise it is not so.

\begin{corollary}
System \eqref{1}, \eqref{E} has  infinite series of conservation laws:
\begin{eqnarray}\label{CL}
&&\left(M_i+\frac{i}{2} M_{i-1} E^2\right)_t+\left(M_{i+1} +\frac{i}{2} M_{i} E^2 U_i -i M_{i-1} E\right)_x=0,\quad i\ge 1.
\end{eqnarray}
In turn,
system (MF), \eqref{E} has two infinite series of conservation laws:
\begin{eqnarray*}\nonumber
&&\left(n_k+(k-1) n_{k-1} E^2\right)_t+\left(n_k U_k+(k-1) n_{k-1} U_{k-1} E^2- 2(k-1) n_{k-1} E\right)_x=0,\\
&&\left(n_k U_k+\frac{2k-1}{2} n_{k} E^2\right)_t+\left(n_{k+1} +\frac{2k-1}{2} n_{k} E^2 U_k -(2k-1) n_k E\right)_x=0,
\end{eqnarray*}
$k\in\mathbb N$, $k\ge 2.$

\end{corollary}
The {\it proof} is a computation. To obtain \eqref{CL} we use $(E^2)_t M_{i-1} + (E^2)_x M_i= 2 E M_i$. $\Box$

\medskip

 \subsection{ Examples of two first closures}
{{\bf 1.} $k=1$, $j=2k-1=1$. } Here \begin{eqnarray*}
A=\begin{pmatrix}
0 & 1 \\ -\frac{M_3^2}{M_{2}^2} &\frac{2 M_4}{M_{2}}
\end{pmatrix}.
\end{eqnarray*}
The system in the terms of moments is
 \begin{eqnarray*}\label{6}
&& E_t + \frac{M_1}{M_0} E_x=  \frac{M_1}{M_0}, \quad
  (M_0)_t+(M_1 )_x=0,\quad
 (M_1)_t+\left(\frac{M_1^2}{M_0}\right)_x=- E M_0,
  \end{eqnarray*}
in hydrodynamic terms it is
\begin{eqnarray*}\label{61}
&& E_t + U_1 E_x= U_1, \quad
  (n_1)_t+(n_1 U_1)_x=0,\quad
  (n_1 U_1)_t+(n_1 U_1^2)_x=- E n_1.
  \end{eqnarray*}
 which on smooth solutions is equivalent to the hydrodynamic equations of cold plasma \cite{Ch_book}
  \begin{eqnarray}\label{7}
 && E_t + U_1 E_x= U_1, \qquad
 (U_1)_t+U_1 (U_1)_x=- E.
  \end{eqnarray}

{ {\bf 2.} $k=2$, $j=2 k-1=3$}. Here \begin{eqnarray*}
A=\begin{pmatrix}
0 & 1 & 0 & 0\\
0 & 0 & 1& 0\\
0 & 0 & 0& 1\\
0 & 0&  -\frac{M_3^2}{M_{2}^2} &\frac{2 M_4}{M_{2}}
\end{pmatrix},
\end{eqnarray*}
System in the terms of moments is
   \begin{eqnarray}\label{k2M_1}
 &&E_t + \frac{M_1}{M_0} E_x= \frac{M_1}{M_0},\qquad
  (M_0)_t+(M_1)_x=0,\qquad
  (M_1)_t+(M_2)_x=- E M_0,\\
 &&  (M_2)_t+(M_3)_x=- 2 E M_1,\qquad
  (M_3)_t+\left(\frac{M_3^2}{M_2}\right)_x=- 3 E M_2.\label{k2M_2}
  \end{eqnarray}
and in hydrodynamic terms is
\begin{eqnarray}\label{k2MF1}
 && E_t + U_1 E_x= U_1, \quad
  (n_1)_t+(n_1 U_1)_x=0,\quad
  (n_1 U_1)_t+(n_2)_x=- E n_1,\\
&&  (n_2)_t+(n_2 U_2)_x=- 2 E n_1 U_1,\qquad
  (n_2 U_2)_t+(n_2 U_2^2)_x=- 3 E n_2.\label{k2MF2}
  \end{eqnarray}

\section{Expression for heat flux}

In the traditional method of closing the Vlasov-Poisson equations they use the  moments up to the third order, e.g. \cite{Huang}. Namely, through  $M_0$ and $M_1$ the velocity $U=U_1=\frac{M_1}{M_0}$ is  determined, and then the centered moments of the second and third order are considered.
The centered moment of the second order $p(t,x)=\int\limits_{\mathbb R}(v-U)^2 F dv$ corresponds to the pressure, it is easy to see that   $p=P_1$ according to \eqref{nUP_MF}. The third centered moment is the heat flux $q(t,x)=\int\limits_{\mathbb R}(v-U)^3 F dv$. The resulting system consists of three equations \eqref{k2MF1} and the additional equation
\begin{eqnarray}\label{q}
 && p_t + U p_x+ 3 p U_x= -q_x,
   \end{eqnarray}
\eqref{k2MF1}, \eqref{q} is not closed and the problem is how to express $q$ through $n=M_0, U$ and $p$.
It is can be readily computed that under our approach
\begin{eqnarray*}
 && q= M_3+\frac{3 M_1 M_2}{M_0} - \frac{4 M_1^3 }{M_0^2} = n_2 (3 U_1+U_2) - 4 U_1^3 n_1.
   \end{eqnarray*}
Thus the heat flux can be expressed through the solution of system \eqref{k2M_1}, \eqref{k2M_2} or, alternatively, system \eqref{k2MF1}, \eqref{k2MF2}. Moreover, one can obtain $M_0,\dots, M_3$ from the closure higher than $k=2$  and use only first components of  $M_0, M_1,\dots, M_j$.

\section{Closure based on the assumption about the distribution function $F(t,x,v)$.}
A priori assumptions about the properties of the distribution function sometimes help close the moment chain \cite{Kuehn}.

{\bf 1.} As noted above, if $
F(t,x,v)=P(t,x) \,\delta(v-v_0(t,x))
$, then $P_1=0$ and the moment equations \eqref{1}  are closed at the first step and system (MF) reduces to the cold plasma equations \eqref{7}.

{\bf 2.} System \eqref{1} was written out in \cite{Jin}, \cite{Jin1},  where it was used to construct a numerical method for multivalued solutions, but its closure was made under the assumption that
$
F(t,x,v)=\sum\limits_{i=1}^j P_i(t,x)\delta(v-v_i(t,x)),
$
where $(P_i, v_i)$ are the density and velocity of the $i$th branch. Note that the  sense of $(P_i, v_i)$ is different from $(n_i,U_i)$ introduced in Sec.\ref{MFI}.

\label{principle}

{\bf 3.}
 If
$
F(t,x,v)=P(t,x)\, {\mathcal K}(v-v_0(t,x)),
$
where ${\mathcal K}(v-v_0)$ is the heat kernel, the density of the Gaussian distribution (or any other symmetric distribution), then
all moments can be expressed through expectation $\mu$ and dispersion $\sigma^2$ as
$$M_1=\mu, \, M_2=\mu^2+\sigma^2,\, M_3=\mu^3+3 \mu \sigma^2,\, M_4=\mu^4+6 \mu^2 \sigma^2+3\sigma^4,\, \dots.$$
The result of the closure is (in variables $n=n_1$, $U=U_1$, $p=P_1$)
\begin{eqnarray*}
&&E_t + U E_x= U,\quad
n_t+(n U)_x=0,\quad
(n U)_t+(n U^2 +P)_x=- E n, \quad
p_t+U p_x + 3 p  U_x=0,
\end{eqnarray*}
which implies $p= e^S n^3,$ \,  $\frac{dS}{dt}=0.$ This system corresponds to
the isentropic gas dynamics of one-atomic gas, $q=0$.

{\bf 4.} There are many other distributions whose moments are completely determined by the first few moments. One such example is the Pearson distribution, which contains numerous special cases. In general, the Pearson distribution is determined by the first four moments. In each of these cases, it is possible to obtain a closure of the moment chain.

\section{Discussion}
\subsection{The moment problem.} If all moments are known, we can {\it formally} reconstruct $F$ as follows:
\begin{equation*}\label{FM}
F(t,x,v)=\sum\limits_{i=0}^\infty \, \frac{(-1)^i}{i!} \,\delta^{(i)}(v) \, M_i(t,x).
\end{equation*}
However, for the sequence of moments to uniquely determine the distribution function, it must satisfy a specific condition \cite{Lin}. One sufficient condition for uniqueness is
$
M_{2i}= O (2i)^{2i}, \, i\to \infty
$,
which implies, in particular, that the even moments grow rapidly. The moments of a chain truncated by setting terms to zero starting from a certain step (a common practice in physical problems \cite{Kuehn}) do not satisfy this condition. Although this condition is merely sufficient, there are known examples demonstrating that the distribution cannot be uniquely reconstructed from such a finite set of moments. A natural question arises: does the closure principle \eqref{principle} contradict the sufficient conditions for the unique determination of the distribution function from its moments? It turns out that this principle does not guarantee uniqueness either. Indeed, one can show that the assumption \eqref{principle} implies that one possible form of the distribution function for closure at step $k$ is $F(t,x,v)=\sum\limits_{i=1}^k n_i(t,x)\delta(v-U_i(t,x)).$ For such a distribution, $M_{2i}=\sum\limits_{j=1}^k n_j U_j^{2i}\sim U_{max}^{2i} $, where $U_{max}=\max\limits_{j} |U_j|$.
Thus, the sufficient condition is not met. Examples confirm the existence of an infinite number of continuous distributions possessing the same moments as the discrete distribution $F(t,x,v)$ mentioned above \cite{Akhiezer}.

\subsection{The problem of the smoothness of the solution to the Cauchy problem}

 Iordanskii \cite{Iord} considered the Cauchy problem for system \eqref{VP} and showed that under some fairly general assumptions about the electron density
 $F(t,x,v)$ the solution preserves global smoothness for all initial data. For the hydrodynamic model of cold plasma \eqref{7}, the first step of the closure, on the contrary, there is a wide class of initial data for which the derivatives of solution go to infinity  in a finite time, see \cite{RCh} for the criterion of the singularities formation.
It would be natural to expect that with closure at each subsequent step the requirements on the initial data that ensure the global smoothness of the solution are weakened.

However, preliminary numerical results indicate that shifting the closure to the next step does not automatically extend the system's lifespan; the initial data must be specially adapted for this purpose. How to properly carry out this adaptation remains an open question. In effect, applying the closure at each subsequent step should impart certain "parabolic" characteristics to the hyperbolic system.

In this context, it is worth noting that a system closure method discovered via machine learning, which reproduces Landau damping, results in a following  degenerate parabolic system
\begin{eqnarray*}
&E_t + U E_x= U,\qquad
&n_t+(n U)_x=0,\\
&(n U)_t+(n U^2 +nT)_x=- E n, \qquad
&T_t+U T_x + 2 T U_x= q (\sqrt{T} T_x)_x,
\end{eqnarray*}
$q=4.2665$ according to one study and $q=4.1231$ according to another, independent one \cite{Burles}, \cite{HuangZ}.
Here nonlinear diffusion applies to only one variable, temperature $T=np$.

\subsection{Generalization to the three-dimensional case}

In principle, the multicomponent medium model can be extended to the three-dimensional case.  The main obstacle is the loss of a strict ordering in the hierarchy of moments, which is a direct consequence of the tensor nature of moments in three-dimensional space. In the absence of canonical ordering, there is no canonical closure either; consequently, the method loses the very feature that made it attractive in the one-dimensional case: a unified, mathematically natural truncation principle.



\end{document}